\documentclass[12pt,oneside,reqno,bibliography=totoc]{scrartcl}

\usepackage{cihan}

\AtBeginDocument{%
  \addtolength\abovedisplayskip{-0.5\baselineskip}%
  \addtolength\belowdisplayskip{-0.5\baselineskip}%
}

\title{Polynomial conditions and homology of \texorpdfstring{$\FI$}{Lg}-modules}
\author{Cihan Bahran}
\affil{Department of Mathematics, Bo\u{g}azi\c{c}i University\\
Bebek, 34342 Istanbul, Turkey}\date{}
\date{}

\newcommand{\FI}{\mathbf{FI}}

\DeclareMathOperator{\sr}{\mathbf{st-rank}}

\DeclareMathOperator{\crit}{crit}

\DeclareMathOperator{\FB}{\mathbf{FB}}

\newcounter{dummy}
\makeatletter
\newcommand\sitem[1][]{\item[(#1)]\refstepcounter{dummy}\def\@currentlabel{#1}}
\makeatother

\newcommand{\cofi}[1]{\co_{#1}^{\FI}}

\DeclareMathOperator{\ab}{\mathcal{A}}

\DeclareMathOperator{\reg}{reg}

\DeclareMathOperator{\induce}{\Ind_{\FB}^{\FI}}

\newcommand{\fU}{\mathfrak{U}}

\newcommand{\tgen}{t_{0}}
\newcommand{\trel}{t_{1}}
\newcommand{\shift}[2]{\mathbf{\Sigma}^{#2}   #1 }

\newcommand{\weak}{\delta}

\newcommand{\gaga}{\mathcal{G}}

\newcommand{\deriv}{\mathbf{\Delta}}
\newcommand{\kert}{\mathbf{K}}

\newcommand{\bul}{\bullet}

\newcommand{\hell}{h}
\newcommand{\local}{\hell^{\text{max}}}

\DeclareMathOperator{\poly}{\mathbf{Poly}}

\newcommand{\locoh}[1]{\co_{\mathfrak{m}}^{#1}}

\AtBeginDocument{%
   \def\MR#1{}
}

\makeatletter
\def\blfootnote{\gdef\@thefnmark{}\@footnotetext}
\makeatother

\begin{document}
\maketitle

\blfootnote{\textup{2010} \textit{Mathematics Subject Classification}.
Primary 18A25, 20J06; Secondary 11F75.} 
\blfootnote{The author was supported in part by T\"{U}B\.{I}TAK 119F422.}
\vspace{-0.8 in}
\begin{onecolabstract} 
We identify two recursively defined polynomial conditions for $\FI$-modules in the literature. We characterize these conditions using homological invariants of $\FI$-modules (namely the local degree and regularity, together with the stable degree) and clarify their relationship. For one of these conditions, we give improved twisted homological stability ranges for the symmetric groups. As another application, we improve the representation stability ranges for congruence subgroups with respect to the action of an appropriate linear group by a factor of 2 in its slope.
\end{onecolabstract}
\vspace{0.4 in}

{\tableofcontents}

\section{Introduction}
There are (at least) two classes of papers that deal in some depth with $\FI$-modules: 
\begin{birki}
 \item There are papers such as \cite{cefn}, \cite{ce-homology}, \cite{li-ramos}, \cite{nss-regularity}, \cite{cmnr-range} where the $\FI$-module is the central object of study. They attach \textbf{homological invariants} to an $\FI$-module by means such as $\FI$-homology or local cohomology, and study the relationship of these invariants both with the stabilization behavior of the $\FI$-module and/or between each other.
 \item There are papers such as \cite{rw-wahl-stab}, \cite{patzt-central}, \cite{mpw-torelli-H2}, \cite{mpp-secondary}, \cite{putman-poly}, which might be thought of as stability machines. The sequence $\{\sym{n}\}$ of symmetric groups is but one of many sequences of groups they deal with, and $\FI$-modules arise as the suitable notion of \textbf{coefficient systems} for $\{\sym{n}\}$. They declare a coefficient system to be \textbf{polynomial} with certain parameters in a \textbf{recursive} fashion: there is a base case, and above that, being polynomial demands a related coefficient system to be polynomial with some of the parameters lowered.
\end{birki}
The main objective of this paper is to characterize the polynomial conditions in (2) for $\FI$-modules by the homological invariants in (1). 

\paragraph{Notation.} We write $\FI$ for the category of finite sets and injections. An $\FI$-module is a functor $V \colon \FI \rarr \lMod{\zz}$ and given a finite set $S$, we write $V_{S}$ for its evaluation; given an injection of finite sets $\alpha\colon S \emb T$, we write $V_{\alpha} \colon V_{S} \rarr V_{T}$ for its induced map. For $n \in \nn$ we set $V_{n} := V_{\{1,\dots,n\}}$. We write $\lMod{\FI}$ for the category of $\FI$-modules. Throughout, our notation for $\FI$-modules will be consistent with \cite{cmnr-range} and \cite{bahran-reg}.

\paragraph{Degree and torsion.} Given an $\FI$-module $W$, we write 
\begin{align*}
 \deg(W) &:= \min\{d \geq -1 : W_{S} = 0 \text{ for } |S| > d\} 
 \\
 &\in \{-1,0,1,2,3,\dots\} \cup \{\infty\} \, .
\end{align*}
An $\FI$-module $V$ is \textbf{torsion} if for every finite set $S$ and $x \in V_{S}$, there exists an injection $\alpha \colon S \emb T$ such that $V_{\alpha}(x) = 0 \in V_{T}$. We write
\begin{align*}
 \locoh{0} \colon \lMod{\FI} \rarr \lMod{\FI}
\end{align*}
for the functor which assigns an $\FI$-module its largest torsion $\FI$-submodule, and write 
\begin{align*}
 h^{0}(V) := \deg(\locoh{0}(V)) \, .
\end{align*}

\paragraph{Shift and derivative functors.} Given any $\FI$-module $V$, we write $\shift{V}{}$ for the composition 
\begin{align*}
 \FI \xrightarrow{- \sqcup \{*\}} \FI \xrightarrow{V} \lMod{\zz} \, .
\end{align*}
and call it the \textbf{shift functor}. It receives a natural transformation from the identity functor $\id_{\lMod{\FI}}$, whose cokernel
\begin{align*}
 \deriv := \coker \left( \id_{\lMod{\FI}} \rarr \shift{}{} \right)
\end{align*}
we call the \textbf{derivative functor}.

\paragraph{Stable degree.} For an $\FI$-module $V$, we set
\begin{align*}
 \weak(V) &:= \min\{r \geq -1 : \deriv^{r+1}(V) \text{ is torsion}\}
 \\
 &\in \{-1,0,1,\dots\} \cup \{\infty\} \, ,
\end{align*}
and call it the \textbf{stable degree} of $V$. In both polynomial conditions for $\FI$-modules we shall consider, the stable degree will be in analogy with the usual degree of a polynomial. Also see \cite[Proposition 2.14]{cmnr-range}.

\subsection{First polynomial condition and local degree}

Suppose $f$ is a function in $n \in \nn$ which is equal to a polynomial of degree $ \leq r$ in the range $n \geq L$. We can consider its \textbf{discrete derivative} $\deriv\! f$, which is the function
\begin{align*}
 \deriv\! f(n) := f(n+1) - f(n) \, .
\end{align*}
Note that $\deriv\! f$ is equal to a polynomial of degree $\leq r-1$ in the same range $n \geq L$. The first polynomial condition we treat for $\FI$-modules is a categorification of this recursion. See \cite[Section 4.4 and Remark 4.19]{rw-wahl-stab} for references to similar definitions in the literature.

\begin{defn} \label{defn:poly1}
 For every pair of integers $r \geq -1$, $L \geq 0$, we define a class of $\FI$-modules $\poly_{1}(r,L)$ recursively via 
\begin{align*}
\poly_{1}(r,L) &:= 
\begin{cases}
\left\{ V \in \lMod{\FI} : \deg(V) \leq L - 1 \right\} & \text{\!\!if $r = -1$,}
\vspace{0.4cm}
\\
 \left\{ V \in \lMod{\FI} : 
\begin{array}{l}
 h^{0}(V) \leq L-1 \text{, and}
 \\
 \deriv V \in \poly_{1}\!\big( r-1,L \big)\!\!
\end{array}
\right\} & \text{\!\!if $r \geq 0$.}
\end{cases}
\end{align*}
\end{defn}
\begin{rem} \label{compare-1}
 Let $V$ be an $\FI$-module and $r \geq -1$, $L \geq 0$ be integers. The following can be seen to be equivalent by inspection: 
\begin{itemize}
 \item $V \in \poly_{1}(r,L)$.
 \item In the sense of \cite[Definition 4.10]{rw-wahl-v3}\footnote{Note that \cite{rw-wahl-v3} is an early preprint version of the published \cite{rw-wahl-stab}. The authors switched from Definition \ref{defn:poly1} to Definition \ref{defn:poly2} in between.}, $V$ has degree $r$ at $L$.
 \item In the sense of \cite[Definition 3.24]{mpw-torelli-H2} and \cite[Definition 7.1]{patzt-central}, $V$ has polynomial degree $\leq r$ in ranks $> L-1$.
\end{itemize}
\end{rem}

\paragraph{Local cohomology and local degree.}
 The functor $\locoh{0}$ defined in the beginning above is left exact. For each $j \geq 0$, we write $\locoh{j} := \operatorname{R}^{j}\!\locoh{0}$ for the $j$-th right derived functor of $\locoh{0}$, and write 
\begin{align*}
 h^{j}(V) &:= \deg(\locoh{j}(V)) 
 \\
 &\in \{-1,0,1,\dots\} \cup \{\infty\} \, ,
 \\
 \local(V) &:= \max\{h^{j}(V) : j \geq 0\}
 \\
 &\in \{-1,0,1,\dots\} \cup \{\infty\} \, ,
\end{align*}
for every $\FI$-module $V$. We call $\local(V)$ the \textbf{local degree} of $V$.

Our first main result is that the stable degree $\weak(V)$ and the local degree $\local(V)$ together characterize the first polynomial condition.
\begin{thmx}\label{thm:hmax}
 For every pair of integers $r \geq -1$, $L \geq 0$, we have 
\begin{align*}
 \poly_{1}(r,L) = \big\{
 V \in \lMod{\FI} : 
 \weak(V) \leq r \text{ and } \local(V) \leq L-1
 \big\} \, .
\end{align*}
\end{thmx}

\subsection{Second polynomial condition and regularity}
The second polynomial condition we shall treat is, perhaps deceivingly, very similar to the first one. In fact the confusion between the two and the resulting need to clarify was what prompted this paper. 
\begin{defn} \label{defn:poly2}
 For every pair of integers $r \geq -1$, $M \geq 0$, we define a class of $\FI$-modules $\poly_{2}(r,M)$ recursively via 
\begin{align*}
\poly_{2}(r,M) &:= 
\begin{cases}
\left\{ V \in \lMod{\FI} : \deg(V) \leq M - 1 \right\} & \text{\!\!if $r = -1$,}
\vspace{0.4cm}
\\
 \left\{ V \in \lMod{\FI} : 
\begin{array}{l}
 h^{0}(V) \leq M-1 \text{, and}
 \\
 \deriv V \in \poly_{2}\!\big( r-1,\max\{0,M-1\} \big)\!\!
\end{array}
\right\} & \text{\!\!if $r \geq 0$.}
\end{cases}
\end{align*}
\end{defn}

\begin{rem} \label{compare-2}
 Let $V$ be an $\FI$-module and $r \geq -1$, $M \geq 0$ be integers. The following can be seen to be equivalent by inspection: 
\begin{itemize}
 \item $V \in \poly_{2}(r,M)$.
 \item In the sense of \cite[Definition 4.10]{rw-wahl-stab}, $V$ has degree $r$ at $M$.
\item In the sense of \cite[Definition 2.40]{mpp-secondary}\footnote{Although the \emph{terminology} used for the polynomial conditions in \cite{mpw-torelli-H2}, \cite{patzt-central} and \cite{mpp-secondary} are the same, the first two use Definition \ref{defn:poly1} while the latter uses Definition \ref{defn:poly2}.}, $V$ has polynomial degree $\leq r$ in ranks $> M-1$. 
 \item In the sense of \cite[Definition 1.6]{putman-poly}, $V$ is polynomial of degree $r$ starting at $M$.
\end{itemize}
\end{rem}

\paragraph{$\FI$-homology and regularity.} \label{defn:reg}
Consider the functor $\cofi{0} : \lMod{\FI} \rarr \lMod{\FI}$ defined via 
\begin{align*}
 \cofi{0}(V)_{S} := \coker\! \left(
 \bigoplus_{T \subsetneq S} V_{T} \rarr V_{S}
 \right)
\end{align*}
for every finite set $S$, which is right exact. For each $i \geq 0$, we write $\cofi{i} := \operatorname{L}_{i}\!\cofi{0}$ for its $i$-th left derived functor, and write 
\begin{align*}
 t_{i}(V) &:= \deg\!\left( \cofi{i}(V) \right) 
 \\
 &\in \{-1,0,1,\dots\} \cup \{\infty\} \, ,
 \\
 \reg(V) &:= \max\{t_{i}(V) - i : i \geq 1\} \\
  &\in\{-2,-1,0,1,\dots\} \cup\{\infty\} \, ,
\end{align*}
for every $\FI$-module $V$. We say that $V$ is \textbf{generated in degrees} $\leq g$ if $\tgen(V) \leq g$, and that $V$ is \textbf{presented in finite degrees} if $\tgen(V)$ and $\trel(V)$ are both finite. We call $\reg(V)$ the \textbf{regularity} of $V$.

Our second main result is that the stable degree $\weak(V)$ and the regularity $\reg(V)$ together characterize the second polynomial condition.

\begin{thmx}\label{thm:reg}
 For every pair of integers $r \geq -1$, $M \geq 0$, we have 
\begin{align*}
 \poly_{2}(r,M) = \big\{ V \in \lMod{\FI} :
 \weak(V) \leq r \text{ and } \reg(V) \leq M-1
 \big\} \, .
\end{align*}
\end{thmx}

%\begin{rem}
% Note that Theorem \ref{thm:reg} together with \cite[Theorem 3.19]{ramos-fig} and \cite[Proposition 2.9, part (4)]{cmnr-range} imply \cite[Proposition 4.18]{rw-wahl-stab}.
%\end{rem}
%\begin{rem}
% Given an $\FI$-module $V$ and integers $a,b \geq -1$ with $\tgen(V) \leq a$ and $\trel(V) \leq b$, we have 
%\begin{align*}
% \reg(V) \leq 
%\begin{cases}
% -1 & \text{if $a=-1$ or $b \leq 0$,} \\
% a+b-1 & \text{if $0 \leq a \leq b-1$,} \\
% 2b-2 & \text{if $1 \leq b \leq a$,}
%\end{cases}
%\end{align*}
% 
% 
%\begin{align*}
% V \in 
%\begin{cases}
% \poly_{2}(a,0) & \text{if $a=-1$ or $b \leq 0$,} \\
% \poly_{2}(a,a+b) & \text{if $0 \leq a \leq b-1$,} \\
% \poly_{2}(a,2b-1) & \text{if $1 \leq b \leq a$,}
%\end{cases}
%\end{align*}
%\end{rem}

\subsection{Twisted homological stability with $\FI$-module coefficients}
%\begin{thm}[{\cite[Theorem A]{putman-poly}}] \label{putman-range}
% Let $V$ be an $\FI$-module presented in finite degrees and $r,M \geq 0$ be integers such that $V \in \poly_{2}(r,M)$. Then for every $k \geq 0$, the map
%\begin{align*}
% \co_{k}(\sym{n};V_{n}) \rarr \co_{k}(\sym{n+1};V_{n+1})
%\end{align*}
%is an isomorphism for 
%\begin{align*}
% n \geq 2k + \max\{r,M-1\} + 2 \, ,
%\end{align*}
%and a surjection for 
%\begin{align*}
%n \geq 2k + \max\{r,M-1\} + 1 \, .
%\end{align*}
%\end{thm}
For any $\FI$-module $V$ and homological degree $k \geq 0$, there is a sequence of maps 
\begin{align*}
 \co_{k}(\sym{0};V_{0}) \rarr \co_{k}(\sym{1};V_{1}) \rarr \co_{k}(\sym{2};V_{2}) \rarr \cdots
\end{align*}
between the homology groups of the symmetric groups twisted by $V_{n}$'s. For the stabilization of this sequence, most recently Putman \cite[Theorem A, Theorem A']{putman-poly} established explicit ranges for the class $\poly_{2}(r,M)$ in terms of $r,M$. We give ranges for the class $\poly_{1}(r,L)$ in terms of $r,L$.

\begin{thmx} \label{twisted-ranges}
 Let $V$ be an $\FI$-module and $r,L \geq 0$ be integers such that $V \in \poly_{1}(r,L)$. Then for every $k \geq 0$, the map
$\co_{k}(\sym{n};V_{n}) \rarr \co_{k}(\sym{n+1};V_{n+1})$ is 
\begin{itemize}
 \item an isomorphism for $
 n \geq 
\begin{cases}
 2k+r+1 & \text{if $L = 0$,}
 \\
 2k + r + \floor*{\frac{L+1}{2}} + 2 & \text{if $1 \leq L \leq 2r-2$,}
 \\ 
 \max\{2k+2r+1,\,L\}& \text{if $L \geq \max\{1,2r-1\}$,}
\end{cases}
 $
 \vspace{0.3cm}
\item and a surjection for $
  n \geq 
\begin{cases}
 2k+r & \text{if $L = 0$,}
 \\
 2k + r + \floor*{\frac{L+1}{2}} + 1 & \text{if $1 \leq L \leq 2r-2$,}
 \\ 
 \max\{2k+2r,\,L\}& \text{if $L \geq \max\{1,2r-1\}$.}
\end{cases}
 $
\end{itemize}
\end{thmx}

\begin{rem} Under the same hypotheses with Theorem \ref{twisted-ranges}, \cite[Theorem 5.1]{rw-wahl-v3} establishes
\begin{itemize}
 \item an isomorphism for $n \geq \max\{2L+1,2k+2r+2\}$,
 \vspace{0.1cm}
 \item and a surjection for $n \geq \max\{2L+1, 2k+2r\}$.
\end{itemize}
The ranges in Theorem \ref{twisted-ranges} are improvements over these.
\end{rem}

\subsection{$\operatorname{SL}_{n}^{\fU}$-stability ranges for congruence subgroups}

For every ring $R$, the assignment $n \mapsto \GL_{n}(R)$ defines an $\FI$-group (a functor from $\FI$ to the category of groups), for which we write $\GL_{\bul}(R)$. If $I$ is an ideal of $R$, as the  kernel of the mod-$I$ reduction we get a smaller $\FI$-group 
\begin{align*}
 \GL_{\bul}(R,I) := \ker \big(\!\GL_{\bul}(R) \rarr \GL_{\bul}(R/I) \big) \, ,
\end{align*}
called the $I$\textbf{-congruence subgroup} of $\GL_{\bul}(R)$.
For each $k \geq 0$ and abelian group $\ab$, taking the $k$-th homology with coefficients in $\ab$ defines an $\FI$-module $\co_{k}(\GL_{\bul}(R,I);\ab)$. 

We wish to extend the $\sym{n}$-action on $\co_{k}(\GL_{n}(R,I);\ab)$ to an action of a linear group and formulate representation stability over it, in accordance with \cite[fifth remark]{putman-congruence}.

\paragraph{Special linear group with respect to a subgroup of the unit group.} For a commutative ring $A$ and a subgroup $\fU \leq A^{\times}$, we write
\begin{align*}
 \SL_{n}^{\fU}(A) := \{f \in \GL_{n}(A) : \det(f) \in \fU\} \, ,
\end{align*}
so that we interpolate between $\SL_{n}(A) \leq \SL_{n}^{\fU}(A) \leq \GL_{n}(A)$  as we vary $1 \leq \fU \leq A^{\times}$. Note that we are using the notation in \cite{putman-sam}, whereas in \cite{mpw-torelli-H2} and \cite{mpp-secondary} this group is denoted $\GL_{n}^{\fU}(A)$.

\begin{hyp} \label{surj-hyp}
 In the triple $(R,I,n_{0})$, we have a commutative ring $R$, an ideal $I$ of $R$, and an integer $n_{0} \in \nn$ such that the mod-$I$ reduction  
\begin{align*}
  \SL_{n}(R) \rarr \SL_{n}(R/I)
\end{align*}
for the special linear group is surjective for every $n \geq n_{0}$. 
\end{hyp}

\paragraph{Stable rank of a ring.} Let $R$ be a nonzero unital (associative) ring. A column vector $\mathbf{v} \in \mat_{m \times 1}(R)$ of size $m$ is \textbf{unimodular} if there is a row vector $\mathbf{u} \in \mat_{1 \times m}(R)$ such that $\mathbf{u}\mathbf{v} = 1$. Writing $\mathbf{I}_{r} \in \mat_{r \times r}(R)$ for the identity matrix of size $r$, we say
a column vector $\mathbf{v}$ of size $m$ is \textbf{reducible} if there exists  $\mathbf{A} \in \mat_{(m-1) \times m}(R)$ with block form  $\mathbf{A} = [\mathbf{I}_{m-1} \,|\, \mathbf{x}]$ such that the column vector $\mathbf{A} \mathbf{v}$ (of size $m-1$) is unimodular. We write $\sr(R) \leq s$ if every unimodular column vector of size $> s$ is reducible. 

\begin{rem} \label{surj-rem}
We make a few observations about Hypothesis \ref{surj-hyp}:
\begin{birki}
%\item We have the following implications:
%\begin{align*}
% \text{$R/I$ is a Eu}&\text{clidean domain}
% \\
% &\Downarrow
% \\
% \text{$\SL_{n}(R/I)$ is generated by el} &\text{ementary matrices for every $n \in \nn$.}
% \\
% &\Downarrow
% \\
% \text{$\SL_{n}(R) \rarr \SL_{n}(R/I)$ is} &\text{ surjective for every $n \in \nn$.}
% \\
% &\Downarrow
% \\
% (R,\, \I,\, \{x + I : x \in R^{\times}\},\, 0) &\text{ satisfies Hypothesis \ref{surj-hyp}.}
%\end{align*}
\item It is straightforward to check that the triple $(R,I,n_{0})$ satisfies \ref{surj-hyp} if and only if setting $\fU := \{x + I : x \in R^{\times}\}$, there is a short exact sequence 
\begin{align*}
 1 \rarr \GL_{n}(R,I) \rarr \GL_{n}(R) \rarr \SL_{n}^{\fU}(R/I) \rarr 1 
\end{align*}
of groups in the range $n \geq n_{0}$ where the epimorphism is the mod-$I$ reduction. Consequently, for every $n \geq n_{0}$ and any coefficients $\ab$, the conjugation $\GL_{n}(R)$-action on the homology groups $\co_{\star}(\GL_{n}(R,I),\ab)$ descends to an $\SL_{n}^{\fU}(R/I)$-action. It is this action for which we will obtain an improved representation stability range.\vspace{0.1cm}
\item For a Dedekind domain $R$ and any ideal $I$ of $R$, the triple $(R,I,0)$ satisfies Hypothesis \ref{surj-hyp}, see \cite[page 2]{clark-order-type}.\vspace{0.1cm}
\item If $\SL_{n}(R/I)$ is generated by elementary matrices for $n \geq n_{0}$, then $(R,I,n_{0})$ satisfies Hypothesis \ref{surj-hyp}.\vspace{0.1cm}
\item If the $K$-group $\operatorname{SK}_{1}(R/I) = 0$ (equivalently, the natural map $\operatorname{K}_{1}(R/I) \rarr (R/I)^{\times}$ is an isomorphism) and $\sr(R/I) \leq s < \infty$, then by (3) and \cite[4.3.8]{hahn-omeara}, the triple $(R,I,s+1)$ satisfies Hypothesis \ref{surj-hyp}.
\end{birki}
\end{rem}

\begin{thmx} \label{congruence-larger-action}
 Let $I$ be a proper ideal in a commutative ring $R$ and $s,n_{0} \in \nn$ such that 
\begin{itemize}
 \item $
 \sr(R) \leq s  
$, and
 \item the triple $(R,I,n_{0})$ satisfies Hypothesis \ref{surj-hyp} with $n_{0} \leq 2s+3$.
\end{itemize}
Then writing 
\begin{align*}
 \fU := \{x + I : x \in R^{\times}\} \, , \quad
 \gaga_{n} := \SL_{n}^{\fU}(R/I) \, ,
\end{align*}
for every homological degree $k \geq 1$ and abelian group $\ab$, there is a coequalizer diagram 
\begin{align*}
 \Ind_{\gaga_{n-2}}^{\gaga_{n}} \co_{k}(\GL_{n-2}(R,I);\ab) \rightrightarrows   
 \Ind_{\gaga_{n-1}}^{\gaga_{n}} \co_{k}(\GL_{n-1}(R,I);\ab)
 \rarr
 \co_{k}(\GL_{n}(R,I);\ab)
\end{align*}
of $\zz\gaga_{n}$-modules whenever 
\begin{align*}
 n \geq 
\begin{cases}
 2s + 5 & \text{if $k =1$,}
 \\
 4k + 2s + 2 & \text{if $k \geq 2$.}
\end{cases}
\end{align*}
\end{thmx}

\begin{rem}
The best stable ranges established previously in the literature under the assumptions (with $n_{0} = 0$) of Theorem \ref{congruence-larger-action} are due to Miller--Patzt--Petersen \cite[proof of Theorem 1.4, page 46]{mpp-secondary}: they obtained the conclusion of Theorem \ref{congruence-larger-action} in the range $n \geq 8k + 4s + 9$.
\end{rem}

\section{Homological algebra of $\FI$-modules}
\subsection{Regularity in terms of local cohomology}
We first recall a characterization of the regularity by Nagpal--Sam--Snowden.
\begin{thm}[{\cite[Theorem 1.1, Remark 1.3]{nss-regularity}}]\label{nss}
 Let $V$ be an $\FI$-module presented in finite degrees which is not $\cofi{0}$-acyclic. Then 
\begin{align*}
 \reg(V) &= \max\{h^{j}(V) + j : \locoh{j}(V) \neq 0\} \\
 &= \max\{h^{j}(V) + j : h^{j}(V) \geq 0\} \, .
\end{align*}
\end{thm}

\begin{rem}
 Under the hypotheses of Theorem \ref{nss}, by \cite[Theorem 2.4]{bahran-reg} and \cite[Corollary 2.13]{cmnr-range}, we have
\begin{align*}
 \empt \neq \{j : \locoh{j}(V) \neq 0\} = \{j : h^{j}(V) \geq 0\} \subseteq \{0,\,\dots,\,\weak(V)+1\}
\end{align*}
\end{rem}

\begin{cor}\label{finite-support}
 Let $V$ be a nonzero $\FI$-module with $\deg(V) < \infty$. Then $V$ is presented in finite degrees, and 
\begin{align*}
 h^{j}(V) = 
\begin{cases}
 \deg(V) = \reg(V) & \text{if $j = 0$,}
 \\
 -1 & \text{otherwise.}
\end{cases}
\end{align*}
\end{cor}
\begin{proof}
 $V$ is certainly generated in degrees $\leq \deg(V)$ and also $h^{0}(V) \leq \deg(V)$. Thus by \cite[Proposition 2.5]{bahran-reg} and \cite[Theorem A]{ramos-coh}, $V$ is presented in finite degrees. Now $V$ and the complex 
$0 \rarr V \rarr 0 \rarr 0 \rarr \cdots$ satisfy the assumptions of \cite[Theorem 2.10]{cmnr-range}, hence 
\begin{align*}
 \locoh{j}(V) = 
\begin{cases}
 V & \text{if $j = 0$,}
 \\
 0 & \text{otherwise.}
\end{cases}
\end{align*}
The rest follows from Theorem \ref{nss}.
\end{proof}

\subsection{The derivative and local cohomology}
In this section, we investigate the relationship between the local cohomology of an $\FI$-module and that of its derivative. 

We write $\kert := \ker (\id_{\lMod{\FI}} \rarr \shift{}{})$, so that we have an exact sequence 
\begin{align*}
 0 \rarr \kert \rarr \id_{\lMod{\FI}} \rarr \shift{}{} \rarr \deriv \rarr 0
\end{align*}
of functors $\lMod{\FI} \rarr \lMod{\FI}$.

\begin{lem}\label{h0-vs-K}
For every $\FI$-module $V$, we have $\deg(\kert V) = h^{0}(V)$.
\end{lem}
\begin{proof}
 Since $\kert V$ is a torsion submodule of $V$, we have $\kert V \subseteq \locoh{0}(V)$ and hence 
\begin{align*}
 \deg(\kert V) \leq \deg(\locoh{0}(V)) = h^{0}(V) \, .
\end{align*}
There is nothing to show when $h^{0}(V) = -1$, so we consider two cases:
\begin{itemize}
 \item $h^{0}(V) = \infty$. To show $\deg(\kert V) = \infty$, we will show that for every $d \in \nn$ we have $\deg(\kert V) \geq d$. Because $h^{0}(V) \geq d$, there exists a torsion element $x \in V_{S} -\{0\}$ of $V$ with $|S| \geq d$. Because $x$ is torsion, the set 
\begin{align*}
 \{|T|: V_{\iota}(x) = 0 \text{ for some } \iota \colon S \emb T\} \subseteq \nn
\end{align*}
is nonempty, hence has a least element, say $N$. Noting $N > d$, let $A$ be a finite set of size $N-1$ and $f \colon S \emb A$ so by the minimality of $N$ we have $0 \neq V_{f}(x) \in (\kert V)_{A}$ and  $\deg(\kert V) \geq |A| = N-1 \geq d$.
 \item $0 \leq d := h^{0}(V) < \infty$: we pick a torsion element $x \in V_{S} -\{0\}$ with $|S| = d$ and we claim that $V_{\iota}(x) = 0$ for the embedding $\iota \colon S \emb S \sqcup \{\star\}$. There is a finite set $T$ and an injection $f \colon S \emb T$ such that $V_{f}(x) = 0$. As $x \neq 0$, $f$ cannot be an isomorphism so $|T| > |S|$ and $f = g \circ \iota$ for some injection $g \colon S \sqcup \{\star\} \emb T$. As $V_{g}(V_{\iota}(x)) = 0$, the element $V_{\iota}(x)$ is torsion but it lies in degree $d+1$, forcing $V_{\iota}(x) = 0$ and hence $x \in \kert V$, showing $\deg(\kert V) \geq d$. 
\end{itemize}
\end{proof}

\begin{prop} \label{fp-deriv}
 Given an $\FI$-module $V$, the following are equivalent: 
\begin{birki}
 \item $V$ is presented in finite degrees.
 \item $h^{0}(V) < \infty$ and $\deriv V$ is presented in finite degrees.
\end{birki}
\end{prop}
\begin{proof}
 Assume (1). Then by \cite[Theorem 1]{gan-shift-seq} $\shift{V}{}$ is presented in finite degrees, and hence so are $\kert V$ and $\deriv V$ by \cite[Theorem B]{ramos-coh} and \cite[Proposition 2.5]{bahran-reg}. We have $h^{0}(V) < \infty$ by \cite[Theorem A]{ramos-coh}.

Conversely, assume (2). By \cite[Proposition 2.9, part (4)]{cmnr-range},  $u := \weak(\deriv V) < \infty$, so $\deriv^{u+2} V =\deriv^{u+1}\deriv V$ is torsion. Also, by applying the implication (1) $\imp$ (2) to $\deriv V$ and iterating it, $\deriv^{u+2} V$ is presented in finite degrees. Being a torsion $\FI$-module generated in finite degrees, $\deriv^{u+2} V$ has finite degree, say $d$. Therefore by \cite[Proposition 4.6]{ce-homology}, $V$ is generated in degrees $\leq u + d + 2$. We conclude by \cite[Theorem A]{ramos-coh}.
\end{proof}
\begin{prop} \label{locoh-LES} Given an $\FI$-module $V$, the following hold:
\begin{birki}
 \item If $h^{0}(V) < \infty$, then there is a long exact sequence 
\begin{align*}
\xymatrixrowsep{0.65cm}
 \xymatrix{
   & & 0 
  \ar[r] & \kert V
  \ar@{->} `r/8pt[d] `/10pt[l] `^dl[ll]|{} `^r/3pt[dll] [dll] 
%  \\ 
%   & \locoh{0}(V) \ar[r] & \shift{\!\locoh{0}(V)}{} 
%  \ar[r] & \locoh{0}(\deriv V)
%  \ar@{->} `r/8pt[d] `/10pt[l] `^dl[ll]|{} `^r/3pt[dll] [dll] 
  \\ 
  & \locoh{0}(V) \ar[r] & \shift{\!\locoh{0}(V)}{} 
  \ar[r] & \locoh{0}(\deriv V) \ar[r] & \cdots
  \\
  & \vdots & \vdots & \vdots
  \\
  \cdots \ar[r] & \locoh{j}(V) \ar[r] & \shift{\!\locoh{j}(V)}{}  
  \ar[r] & \locoh{j}(\deriv V)
 \ar@{->} `r/8pt[d] `/10pt[l] `^dl[ll]|{} `^r/3pt[dll] [dll] 
  \\ 
  & \locoh{j+1}(V) \ar[r] & \shift{\!\locoh{j+1}(V)}{} 
  \ar[r] & \locoh{j+1}(\deriv V) \ar[r] & \cdots
    }
\end{align*}
\vspace{0.1cm}
\item If $V$ is presented in finite degrees, (1) holds such that every $\FI$-module in the sequence has finite degree.
\end{birki}
\end{prop}

\begin{proof}
 For (1), note that $\kert V$ is certainly generated in degrees 
\begin{align*}
  \leq \deg(\kert V) = h^{0}(V) < \infty 
\end{align*}
by Lemma \ref{h0-vs-K}. Thus by \cite[Theorem A]{ramos-coh} (see \cite[Proposition 2.5]{bahran-reg}), $\kert V$ is presented in finite degrees. Therefore \cite[Theorem 2.10]{cmnr-range} applies to $\kert V$ and the complex $0 \rarr \kert V \rarr 0 \rarr 0 \rarr \cdots$, and hence 
\begin{align*}
 \locoh{j}(\kert V) = 
\begin{cases}
 \kert V & \text{if $j = 0$,}
 \\
 0 & \text{otherwise.}
\end{cases}
\end{align*}
Now applying $\locoh{0}$ to the short exact sequence 
\begin{align*}
 0 \rarr \kert V \rarr V \rarr V/\kert V \rarr 0 \, ,
\end{align*}
the associated long exact sequence gives a short exact sequence 
\begin{align*}
 0 \rarr \kert V \rarr \locoh{0}(V) \rarr \locoh{0}(V/\kert V) \rarr 0
\end{align*}
and isomorphisms 
\begin{align*}
 \locoh{j}(V) \cong \locoh{j}(V / \kert V)
\end{align*}
for every $j \geq 1$. Using these isomorphisms after applying $\locoh{0}$ to the short exact sequence 
\begin{align*}
 0 \rarr V / \kert V \rarr \shift{V}{} \rarr \deriv V \rarr 0 \, ,
\end{align*}
the associated long exact sequence will almost have the desired form, except we need to splice it in the beginning and interchange the order of the shift functor $\shift{}{}$ with local cohomology $\locoh{\star}$ in the middle column. 
To see $\shift{}{} \circ \locoh{\star} = \locoh{\star} \circ \shift{}{}$, first note that $\shift{}{} \colon \lMod{\FI} \rarr \lMod{\FI}$ 
\begin{itemize}
 \item is exact,
 \item has an exact left adjoint \cite[Theorem 4]{gan-negative-shift},
 \item satisfies $\shift{}{} \circ \locoh{0} = \locoh{0\,\,} \circ\, \shift{}{}$.
\end{itemize}
Consequently, given an $\FI$-module $U$ and an injective resolution $0 \rarr U \rarr \mathbf{I}^{\star}$, applying $\shift{}{}$ we get an injective resolution 
$0 \rarr \shift{U}{} \rarr \shift{\mathbf{I^{\star}}}{}$ of $\shift{U}{}$, and hence 
\begin{align*}
 \locoh{j}(\shift{U}{}) &= \co^{j}(\locoh{0}(\shift{\mathbf{I^{\star}}}{}))
 = \co^{j}(\shift{\!\locoh{0}(\mathbf{I^{\star}})}{})
 = \shift{\!\co^{j}(\locoh{0}(\mathbf{I^{\star}}))}{} = \shift{\!\locoh{j}(U)}{}
\end{align*}
for every $j \geq 0$, naturally in $U$.

For (2), assume $V$ is presented in finite degrees. Then $\deg(\kert V) = h^{0}(V) < \infty$  (so we have the long exact sequence from (1)) and $\deriv V$ is presented in finite degrees by Proposition \ref{h0-vs-K} and Proposition \ref{fp-deriv}. Now invoke \cite[Theorem 2.10]{cmnr-range} for $V$ and $\deriv V$.
\end{proof}

\begin{cor} \label{locoh-deriv} For every $\FI$-module $V$ presented in finite degrees, the following hold:
\begin{birki}
 \item For every $j \geq 0$, we have $h^{j}(\deriv V) 
 \leq \max\!\left\{ h^{j}(V) - 1,\, h^{j+1}(V) \right\}$.
 \vspace{0.1cm}
\item For every $j \geq 1$, we have $h^{j}(V) \leq \max\{h^{j-1}(\deriv V),h^{j}(\deriv V)\}$.
\end{birki}
\end{cor}
\begin{proof}
 By Proposition \ref{locoh-LES}, for every $j \geq 0$ we have 
\begin{align*}
 h^{j}(\deriv V) &= \deg \locoh{j}(\deriv V) 
 \\
 &\leq 
 \max\!\left\{
 \deg \shift{\!\locoh{j}(V)}{},\,
 \deg\locoh{j+1}(V)
 \right\} =
 \max\!\left\{
 \deg \shift{\!\locoh{j}(V)}{},\,
 h^{j+1}(V)
 \right\} \, .
\end{align*}
If $\locoh{j}(V) \neq 0$, then 
\begin{align*}
 \deg \shift{\!\locoh{j}(V)}{} = \deg \locoh{j}(V) - 1 = h^{j}(V) - 1
\end{align*}
 and (1) follows. If $\locoh{j}(V) = 0$, then $\locoh{j}(\deriv V)$ embeds in $\locoh{j+1}(V)$ and (1) again follows.

To prove (2), fix $j \geq 1$ and set $N := \max\{h^{j-1}(\deriv V), h^{j}(\deriv V)\}$ so by Proposition \ref{locoh-LES}, for every $n > N$ we have an isomorphism 
\begin{align*}
 \locoh{j}(V)_{n} \cong \shift{\!\locoh{j}(V)}{}_{n} = \locoh{j}(V)_{n+1} \, .
\end{align*}
But $\locoh{j}(V)$ has finite degree, therefore the above isomorphisms in the entire range $n > N$ have to be between zero modules, so that $h^{j}(V) = \deg \locoh{j}(V) \leq N$.
\end{proof}

\subsection{Critical index and the regularity of derivative}
In this section, we introduce the notion of critical index for an $\FI$-module and use it to study how regularity interacts with the derivative functor.
\begin{defn} \label{defn:crit}
 For an $\FI$-module $V$ presented in finite degrees which is not $\cofi{0}$-acyclic, we define its \textbf{critical index} as
\begin{align*}
 \crit(V) := 
 \min\!\left\{
 j : h^{j}(V) \geq 0 \text{ and }
 h^{j}(V) + j = \reg(V)
 \right\}
\end{align*}
\end{defn}

\begin{rem}
 Let $V$ be as in Definition \ref{defn:crit}, and set $\gamma:= \crit(V)$, $\rho := \reg(V)$. The following will not be needed in our arguments but we note them for context.
\begin{birki}
 \item   By Theorem \ref{nss} and \cite[Theorem 2.10]{cmnr-range}, the set of indices
\begin{align*}
 \left\{
 j : h^{j}(V) \geq 0 \text{ and }
 h^{j}(V) + j = \rho
 \right\}
\end{align*}
is a nonempty subset of $\{0,\dots,\weak(V)+1\}$; thus $0 \leq \gamma \leq \weak(V) + 1$. 
\vspace{0.1cm}
\item Although they do not give it a name, Nagpal--Sam--Snowden use the critical index: their invariant $\nu$  \cite[Definition 3.3]{nss-regularity} satisfies 
\begin{align*}
 \nu\!\left(
 \cofi{i}(V)_{i+\rho} 
 \right)
 = i + \gamma
\end{align*}
for $i \gg 0$ \cite[Proposition 4.3]{nss-regularity}.
\vspace{0.1cm}
\item It is possible that $h^{\gamma}(V) < \local(V)$. To see this, let us start by an exact sequence 
\begin{align*}
 0 \rarr Z \rarr A \rarr B \rarr W \rarr 0 \tag{$\lozenge$}
\end{align*}
of $\FI$-modules presented in finite degrees where $A,B$ are $\cofi{0}$-acyclic and $\deg (W) = 0$. Breaking this into two short exact sequences, the associated long exact sequences for $\locoh{*}$ yields
\begin{align*}
 \locoh{j}(Z) \cong \locoh{j-2}(W)
\end{align*}
for every $j \geq 0$. In particular $h^{2}(Z) = 0$ and $h^{j}(Z) = -1$ if $j \neq 2$. Now setting $V := Z \oplus T$ with $\deg(T) = 1$, we get 
\begin{align*}
 h^{j}(V) = 
\begin{cases}
 1 & \text{if $j=0$,}
 \\
 0 & \text{if $j=2$,}
 \\ 
 -1 & \text{otherwise.}
\end{cases}
\end{align*}
so that $\reg(V) = \crit(V) = 2$, $h^{2}(V) = 0$, but $\local(V) = 1$.

\end{birki}
\end{rem}
\begin{prop} \label{to-the-deriv}
 Let $V$ be an $\FI$-module presented in finite degrees which is not $\cofi{0}$-acyclic. Then the following hold:
\begin{birki}
 \item $\reg(\deriv V) \leq \reg(V) -1$.
 \vspace{0.1cm}
 \item If $\crit(V) \geq 1$, then $0 \leq \reg(\deriv V) = \reg(V) -1$ and $\crit(\deriv V) = \crit(V) -1$.
\end{birki}
\end{prop}
\begin{proof}
 We note $\deriv V$ is presented in finite degrees by Proposition \ref{fp-deriv}, and set $\rho := \reg(V)$, $\gamma := \crit(V)$. 
 
Assume $\deriv V$ is $\cofi{0}$-acyclic. Then by \cite[Theorem 2.4]{bahran-reg} $h^{j}(\deriv V) = -1$ for every $j \geq 0$, and hence by part (2) of Corollary \ref{locoh-deriv} we have $h^{j}(V) = -1$ for every $j \geq 1$, forcing $\gamma = 0$. Therefore the condition $\reg(\deriv V) < 0$ (which is equivalent to $\deriv V$ being $\cofi{0}$-acyclic by \cite[Corollary 2.9]{bahran-reg}) implies $\gamma = 0$. In this case, we further have 
\begin{align*}
 \reg(\deriv V) < 0 \leq h^{0}(V) = \rho
\end{align*}
by \cite[Theorem 2.4]{bahran-reg} and (1) follows.
 
Next, assume $\deriv V$ is not $\cofi{0}$-acyclic (hence neither is $V$, see the discussion in \cite[Section 2.3]{cmnr-range}). Let us write
\begin{align*}
 J(V) &:= \{j \geq 0: h^{j}(V) \geq 0\} \, .
\end{align*}
Let $j \in J(\deriv V)$. By part (1) of Corollary \ref{locoh-deriv}, we have either $0\leq h^{j}(\deriv V) \leq h^{j}(V) - 1$ or $0 \leq h^{j}(\deriv V) \leq h^{j+1}(V)$. In the former case, we have $j \in J(V)$ and 
\begin{align*}
 h^{j}(\deriv V) + j \leq h^{j}(V) + j - 1 \leq \rho - 1
\end{align*}
by Theorem \ref{nss}, and in the latter case, we have $j+1 \in J(V)$
and
\begin{align*}
 h^{j}(\deriv V) + j \leq h^{j+1}(V) + j + 1 - 1  \leq \rho - 1
\end{align*}
by Theorem \ref{nss}. Yet another application of Theorem \ref{nss} now yields $\reg(\deriv V) \leq \rho - 1$, which is exactly (1). To prove (2), we further assume that $\gamma \geq 1$. We claim that
\begin{align*}
 h^{\gamma-1}(\deriv V) + \gamma - 1 = \rho - 1 \, .
\end{align*}
To that end, by Proposition \ref{locoh-LES} we have an exact sequence
\begin{align*}
  \shift{\!\locoh{\gamma-1}(V)}{} \rarr \locoh{\gamma-1}(\deriv V) \rarr  \locoh{\gamma}(V) \rarr \shift{\!\locoh{\gamma}(V)}{}
\end{align*}
which we evaluate at a finite set of size $\rho-\gamma$ to get an exact sequence
\begin{align*}
  \locoh{\gamma-1}(V)_{\rho - \gamma + 1} 
  \rarr \locoh{\gamma-1}(\deriv V)_{\rho-\gamma} 
  \rarr  \locoh{\gamma}(V)_{\rho-\gamma}
  \rarr \locoh{\gamma}(V)_{\rho-\gamma+1}
\end{align*}
of $\sym{\rho-\gamma}$-modules. Here 
\begin{itemize}
 \item $\locoh{\gamma-1}(V)_{\rho - \gamma + 1} = 0$, because by the definition of critical index we have 
\begin{align*}
 h^{\gamma-1}(V) + \gamma - 1 &< \rho \, ,
 \\
 \deg(\locoh{\gamma-1}(V)) &< \rho - \gamma + 1 \, .
\end{align*}
\item $\locoh{\gamma}(V)_{\rho-\gamma} \neq 0$ and $\locoh{\gamma}(V)_{\rho-\gamma+1} = 0$, because by the definition of critical index we have 
\begin{align*}
 h^{\gamma}(V) \geq 0 \text{ and } h^{\gamma}(V) + \gamma &= \rho \, ,
 \\
 \deg(\locoh{\gamma}(V)) &= \rho - \gamma \, .
\end{align*}
\end{itemize}
Therefore we conclude that 
\begin{align*}
 \locoh{\gamma-1}(\deriv V)_{\rho-\gamma} &\neq 0 \, ,
 \\
 h^{\gamma-1}(\deriv V) \geq 0 &\text{ and }
 h^{\gamma-1}(\deriv V) + \gamma - 1 \geq \rho - 1\, .
\end{align*}
On the other hand, part (1) and Theorem \ref{nss} give the reverse inequality to the above, establishing our claim, the equation $\reg(\deriv V) = \rho - 1$, and the inequality $\crit(\deriv V) \leq \gamma - 1$. To see in fact $\crit(\deriv V) = \gamma - 1$, we can take $0 \leq j < \gamma - 1$ and evaluate the exact sequence
\begin{align*}
  \shift{\!\locoh{j}(V)}{} \rarr \locoh{j}(\deriv V) \rarr  \locoh{j+1}(V) 
\end{align*}
at a finite set of size $\rho - 1 - j$ to get
\begin{align*}
  0 = \locoh{j}(V)_{\rho - j} \rarr \locoh{j}(\deriv V)_{\rho - 1 - j} \rarr  \locoh{j+1}(V)_{\rho - (j + 1)} = 0
\end{align*}
and conclude $h^{j}(\deriv V) + j < \rho - 1$, as desired.
\end{proof}
%So $\reg(\deriv^{\!\gamma}V) = \reg(V) - \gamma$ and $\crit(\deriv^{\!\gamma}V) = 0$.

\subsection{Identifying the polynomial conditions}
In this section we prove Theorem \ref{thm:hmax} and Theorem \ref{thm:reg}. Because $\FI$-modules being presented in finite degrees is such a common assumption, we first incorporate it as a redundant hypothesis in Theorem \ref{thm:hmax-fp} and Theorem \ref{thm:reg-fp}, and then remove this redundancy using Theorem \ref{finite-manifest}. 

\begin{thm} \label{finite-manifest}
 For an $\FI$-module $V$ with $\weak(V) < \infty$, the following are equivalent:
\begin{birki}
 \item $\reg(V) < \infty$.
 \item $\local(V) < \infty$.
 \item $V$ is presented in finite degrees. 
\end{birki}
\end{thm}
\begin{proof}
(3) $\imp$ (1): Immediate from \cite[Theorem A]{ce-homology}. 

(1) $\imp$ (2): Writing 
\begin{itemize}
 \item $\FB$ for the category of finite sets and bijections,
 \vspace{0.1cm}
 \item $\induce$ for the left adjoint of the restriction functor $\Res_{\FB}^{\FI} \colon \lMod{\FI} \rarr \lMod{\FB}$,
 \vspace{0.1cm}
 \item  $W := \cofi{0}(V)$,
\end{itemize}
there is a short exact sequence 
\begin{align*}
 0 \rarr K \rarr \induce(W) \rarr V \rarr 0 \tag{$\dagger$}
\end{align*}
for some $\FI$-module $K$. Here $\induce(W)$ is $\cofi{0}$-acyclic \cite[Lemma 2.3]{ce-homology}. Moreover, the $\cofi{0}$-image of the epimorphism in $(\dagger)$ is the identity map $\cofi{0}(V) \rarr \cofi{0}(V)$. Thus applying $\cofi{0}$ to $(\dagger)$, the associated long exact sequence splits into isomorphisms
\begin{align*}
 \cofi{i+1}(V) &\cong \cofi{i}(K)
\end{align*}
for every $i \geq 0$. In particular we have
\begin{align*}
  t_{i}(K) = t_{i+1}(V) < \reg(V) + i + 1 < \infty
\end{align*}
for every $i \geq 0$, so $K$ is presented in finite degrees. Consequently $\local(K) < \infty$ by \cite[Proposition 2.9, part (4)]{cmnr-range} and \cite[Theorem 2.10]{cmnr-range}. We will be done once we show 
\begin{align*}
 \locoh{*}\!\left( \induce(W) \right) = 0 \, ,
\end{align*}
because applying $\locoh{0}$ to $(\dagger)$, the long exact sequence yields $\local(V) = \local(K)$. The last claim follows from $W$ being the direct product of $\FB$-modules each supported in a single degree, and the functors $\induce$, $\locoh{*}$ commuting with direct products  (for instance, the former via the description \cite[Definition 2.2.2]{cef} and the latter via the description \cite[Definition 5.4]{li-ramos}) together with \cite[Theorem 2.4]{bahran-reg}.

(2) $\imp$ (3): we employ induction on $\weak(V)$: if $\weak(V)=-1$, then $V = \locoh{0}(V)$ is torsion, and so 
\begin{align*}
 \deg(V) = h^{0}(V) \leq \local(V) < \infty \, ,
\end{align*}
thus $V$ is presented in finite degrees by Corollary \ref{finite-support}. 
Next, assume $\weak(V) \geq 0$. We can apply Proposition \ref{locoh-LES} to $V$ to conclude $\local(\deriv V) < \infty$. We also have $\weak(\deriv V) \leq \weak(V)-1$, therefore $\deriv V$ is presented in finite degrees by the induction hypothesis. We conclude by applying Proposition \ref{fp-deriv}.
\end{proof}

\begin{thm}\label{thm:hmax-fp}
 For every pair of integers $r \geq -1$, $L \geq 0$, we have 
\begin{align*}
 \poly_{1}(r,L) = \left\{
 V \in \lMod{\FI} : \begin{tabular}{l}
 $V$ is presented in finite degrees,
 \\
 $\weak(V) \leq r$, and $\local(V) \leq L-1$
\end{tabular}
 \right\} \, .
\end{align*}
\end{thm}
\begin{proof}
 We fix $L \geq 0$ and employ induction on $r$. For the base case $r=-1$, we first let $V \in \poly_{1}(-1,L)$, that is, $\deg(V) \leq L-1$. Then $V$ is torsion so $\weak(V) = -1$, and by Corollary \ref{finite-support} $V$ is presented in finite degrees with $\local(V) \leq L-1$. Conversely, suppose $V$ is presented in finite degrees, $\weak(V) \leq -1$, and $\local(V) \leq L-1$. Then $V$ is torsion, so $\locoh{0}(V) = V$ has degree $\leq L-1$.

For the inductive step, fix $r \geq 0$ and assume that we have
\begin{align*}
 \poly_{1}(r-1,L) = \left\{
 U \in \lMod{\FI} : \begin{tabular}{l}
 $U$ is presented in finite degrees,
 \\
 $\weak(U) \leq r-1$, and $\local(U) \leq L-1$
\end{tabular}
 \right\} \, .
\end{align*} 
Next, let $V \in \poly_{1}(r,L)$, so by Definition \ref{defn:poly1}, $h^{0}(V) \leq L-1$ and 
\begin{align*}
 \deriv V \in \poly_{1}(r-1,L) \, .
\end{align*}
 By the induction hypothesis, we conclude the following:
\begin{itemize}
 \item $\deriv V$ is presented in finite degrees: it follows that $V$ is presented in finite degrees by Proposition \ref{fp-deriv}.
 \item $\weak(\deriv V) \leq r-1$: this means $\deriv^{r}\deriv V = \deriv^{r+1} V$ is torsion, so $\weak(V) \leq r$.
 \item $\local(\deriv V) \leq L-1$: by part (2) of Corollary \ref{locoh-deriv}, we have $\local(V) \leq L-1$.
\end{itemize}
Conversely, let $V$ be an $\FI$-module which is presented in finite degrees, $\weak(V) \leq r$, and $\local(V) \leq L-1$. We observe the following:
\begin{itemize}
 \item $\deriv^{r+1}V = \deriv^{r} \deriv V$ is torsion, so $\weak(\deriv V) \leq r-1$.
 \item By Proposition \ref{fp-deriv}, $\deriv V$ is presented in finite degrees.
 \item $\local(\deriv V) \leq L-1$ by part (1) of Corollary \ref{locoh-deriv}.
\end{itemize}
Therefore by the induction hypothesis, we get $\deriv V \in \poly_{1}(r-1,L)$ and hence $V \in \poly_{1}(r,L)$ by Definition \ref{defn:poly1}.
\end{proof}

\begin{proof}[Proof of \emph{\textbf{Theorem \ref{thm:hmax}}}] Immediate from Theorem \ref{thm:hmax-fp} and Theorem \ref{finite-manifest}.
\end{proof}

\begin{thm}\label{thm:reg-fp}
 For every pair of integers $r \geq -1$, $M \geq 0$, we have 
\begin{align*}
 \poly_{2}(r,M) = \left\{ V \in \lMod{\FI} :
 \begin{tabular}{l}
 $V$ is presented in finite degrees,
 \\
 $\weak(V) \leq r$, and $\reg(V) \leq M-1$
\end{tabular}
 \right\} \, .
\end{align*}
\end{thm}
\begin{proof}
 We fix $M \geq 0$ and employ induction on $r$. For the base case $r=-1$, we first let $V \in \poly_{2}(-1,M)$, that is, $\deg(V) \leq M-1$. Then $V$ is torsion so $\weak(V) = -1$, and by Corollary \ref{finite-support} $V$ is presented in finite degrees with $\reg(V) \leq M-1$. Conversely, suppose $V$ is presented in finite degrees, $\weak(V) \leq -1$, and $\reg(V) \leq M-1$. Then $V$ is torsion, so $\locoh{0}(V) = V$ has degree $\leq M-1$ by Theorem \ref{nss}.

For the inductive step, fix $r \geq 0$ and assume that for every $M' \geq 0$ we have
\begin{align*}
 \poly_{2}(r-1,M') = \left\{ U \in \lMod{\FI} :
 \begin{tabular}{l}
 $U$ is presented in finite degrees,
 \\
 $\weak(U) \leq r-1$, and $\reg(U) \leq M'-1$
\end{tabular}
 \right\} \, .
\end{align*}
Next, fix $M \geq 0$ and let $V \in \poly_{2}(r,M)$, so by Definition \ref{defn:poly2}, $h^{0}(V) \leq M-1$ and 
\begin{align*}
 \deriv V \in \poly_{2}(r-1,\,\max\{0,M-1\}) \, .
\end{align*}
By the induction hypothesis, we conclude the following: 
\begin{itemize}
 \item $\deriv V$ is presented in finite degrees: it follows that $V$ is presented in finite degrees by Proposition \ref{fp-deriv}.
 \item $\weak(\deriv V) \leq r-1$: this means $\deriv^{r}\deriv V = \deriv^{r+1} V$ is torsion, so $\weak(V) \leq r$.
 \item $\reg(\deriv V) \leq \max\{-1,M-2\}$.
\end{itemize}
Three possibilities arise: 
\begin{birki}
 \item $V$ is $\cofi{0}$-acyclic. Then $\reg(V) = -2 \leq M-1$.
\item $V$ is not $\cofi{0}$-acyclic and $\crit(V) = 0$. Here the definition of critical index immediately yields 
\begin{align*}
 \reg(V) = h^{0}(V) \leq M - 1 \, .
\end{align*}
\item $V$ is not $\cofi{0}$-acyclic and $\crit(V) \geq 1$. Part (2) of Proposition \ref{to-the-deriv} yields 
\begin{align*}
 1 \leq \reg(V) = \reg(\deriv V) + 1 \leq \max\{0, M - 1\} \, ,
\end{align*}
hence $M-1 > 0$ and $\reg(V) \leq M-1$.
\end{birki}
Conversely, let $V$ be an $\FI$-module which is presented in finite degrees, $\weak(V) \leq r$, and $\reg(V) \leq M-1$ (in particular $h^{0}(V) \leq M-1$ by Theorem \ref{nss}). We observe the following:
\begin{itemize}
 \item $\deriv^{r+1}V = \deriv^{r} \deriv V$ is torsion, so $\weak(\deriv V) \leq r-1$.
 \item By Proposition \ref{fp-deriv}, $\deriv V$ is presented in finite degrees.
 \item Either $V$ is $\cofi{0}$-acyclic and hence so is $\deriv V$ (see the discussion in \cite[Section 2.3]{cmnr-range}) and $\reg(\deriv V) = -2$, or $V$ is not $\cofi{0}$-acyclic so that
\begin{align*}
 0 \leq \trel(V) - 1 \leq \reg(V) \leq M - 1
\end{align*}
by \cite[Corollary 2.9]{bahran-reg}, and $\reg(\deriv V) \leq M-2$ by part (1) of Proposition \ref{to-the-deriv}. In both cases we have $\reg(\deriv V) \leq \max\{-1,M-2\}$.
\end{itemize}
Therefore the induction hypothesis yields 
$\deriv V \in \poly_{2}(r-1, \max\left\{0,M-1\right\})$. We also have $h^{0}(V) \leq M-1$, so $V \in \poly_{2}(r,M)$ by Definition \ref{defn:poly2}.
\end{proof}

\begin{proof}[Proof of \emph{\textbf{Theorem \ref{thm:reg}}}] Immediate from Theorem \ref{thm:reg-fp} and Theorem \ref{finite-manifest}.
\end{proof}

\subsection{Twisted homological stability}
In this section we prove Theorem \ref{twisted-ranges}.

\begin{proof}[Proof of \textbf{\emph{Theorem \ref{twisted-ranges}}}]
By Theorem \ref{thm:hmax-fp}, $V$ is presented in finite degrees, $\weak(V) \leq r$, and $\local(V) \leq L-1$. Hence by \cite[Theorem 2.6]{bahran-reg}, the triple $(V,\,L-1,\,r)$ satisfies \cite[Hypothesis 1.2]{bahran-reg}. Noting that 
\begin{align*}
 r > \ceil*{\frac{L-1}{2}} \quad \text{if and only if} \quad L < 2r  \, ,
\end{align*}
 by \cite[Theorem C]{bahran-reg} we have 
\begin{align*}
 \reg(V) \leq 
\begin{cases}
 -2 & \text{if $L = 0$,}
 \\
 L & \text{if $L \geq \max\{1,2r\}$,}
 \\
 r + \floor*{\frac{L+1}{2}} & \text{if $1 \leq L < 2r$}.
\end{cases}
\end{align*}
Thus by Theorem \ref{thm:reg}, we have
\begin{align*}
 V \in 
\begin{cases}
 \poly_{2}(r,\,0) & \text{if $L = 0$,}
 \\
 \poly_{2}(r,\,L+1) & \text{if $L \geq \max\{1,2r\}$,}
 \\
 \poly_{2}\left(
 r,\,r + \floor*{\frac{L+1}{2}} + 1
 \right) & \text{if $1 \leq L < 2r$}.
\end{cases}
\end{align*}
Consequently by \cite[Theorem A]{putman-poly}, for every $k \geq 0$ the map
\begin{align*}
 \co_{k}(\sym{n};V_{n}) \rarr \co_{k}(\sym{n+1};V_{n+1})
\end{align*}
is an isomorphism for 
\begin{align*}
 n \geq 
\begin{cases}
 2k+r+1 & \text{if $L = 0$,}
 \\
 2k + L + 2 & \text{if $L \geq \max\{1,2r\}$,}
 \\
 2k + r + \floor*{\frac{L+1}{2}} + 2 & \text{if $1 \leq L < 2r$}.
\end{cases}
\end{align*}
and a surjection for  
\begin{align*}
 n \geq 
\begin{cases}
 2k+r & \text{if $L = 0$,}
 \\
 2k + L + 1 & \text{if $L \geq \max\{1,2r\}$,}
 \\
 2k + r + \floor*{\frac{L+1}{2}} + 1 & \text{if $1 \leq L < 2r$}.
\end{cases}
\end{align*}
It remains to improve the bounds in the case $L \geq \max\{1,2r-1\}$ to 
\begin{itemize}
 \item $n \geq \max\{2k+2r+1,L\}$ for the isomorphism range,
 \item $n \geq \max\{2k+2r,L\}$ for the surjection range.
\end{itemize}
To that end, we induct on $r$. For the base case $r=0$, by \cite[Theorem 2.11]{bahran-reg} there is an $\cofi{0}$-acyclic $I$ with $\weak(I) \leq 0$ and a map $V \rarr I$ which is an isomorphism in degrees $\geq L$. As $\deriv I$ is torsion but also is $\cofi{0}$-acyclic, we have $\deriv I = \kert I = 0$, in other words $I \rarr \shift{I}{}$ is an isomorphism. Thus $I_{n}$ is the same trivial $\sym{n}$-representation for every $n \geq 0$ (namely the abelian group $I_{0}$ with the trivial $\sym{n}$-action). Now by \cite[Corollary 6.7]{nakaoka-Sn}, for every $k \geq 0$ the map
\begin{align*}
 \co_{k}(\sym{n};I_{0}) \rarr \co_{k}(\sym{n+1};I_{0})
\end{align*}
is an isomorphism for $n \geq 2k$. Thus for every $k \geq 0$, the map
\begin{align*}
 \co_{k}(\sym{n};V_{n}) \rarr \co_{k}(\sym{n+1};V_{n+1})
\end{align*}
is an isomorphism for $n \geq\max\{2k,\,L\}$ (which is better than what the base case demands, namely an isomorphism for $n \geq\max\{2k+1,\,L\}$ and a surjection for $n \geq\max\{2k,\,L\}$). Next, take $r \geq 1$ and assume that every $\FI$-module $U \in \poly_{1}(r,L-1)$, that is, by Theorem \ref{thm:hmax-fp}, every $U$ presented in finite degrees with $\weak(U) \leq r-1$ and $\local(U) \leq L - 1$ satisfies\footnote{Here the inequality $L \geq \max\{1,2(r-1)-1\}$ is guaranteed as we are assuming $L \geq \max\{1,2r-1\}$.} the following: for every $k \geq 0$ the map
\begin{align*}
 \co_{k}(\sym{n};U_{n}) \rarr \co_{k}(\sym{n+1};U_{n+1})
\end{align*}
is an isomorphism for 
\begin{align*}
 n \geq\max\{2k+2r-1,\,L\} \, ,
\end{align*}
and a surjection for 
\begin{align*}
 n \geq \max\{2k+2r-2,\,L\} \, .
\end{align*}
In particular by \cite[Proposition 2.9, part (7)]{cmnr-range}, this applies to 
\begin{align*}
 U := \coker (V \rarr \shift{V}{L}) \, .
\end{align*}
In degrees $n \geq L$, writing $I := \shift{V}{L}$, we have a short exact sequence 
\begin{align*}
 0 \rarr V_{n} \rarr I_{n} \rarr U_{n} \rarr 0
\end{align*}
of $\sym{n}$-modules, and the associated long exact sequence in $\co_{*}(\sym{n};-)$ maps to that of $\co_{*}(\sym{n+1};-)$. More precisely, suppressing the symmetric groups in the homology notation, there is a commutative diagram 
\begin{align*}
 \xymatrix{
 \xymatrixcolsep{0.5cm}
 \co_{k+1}(I_{n}) \ar[r] \ar[d]_{\mu_{k+1}} &
 \co_{k+1}(U_{n}) \ar[r] \ar[d]_{\nu_{k+1}} &
 \co_{k}(V_{n}) \ar[r] \ar[d]_{\lambda_{k}} & 
 \co_{k}(I_{n}) \ar[r] \ar[d]_{\mu_{k}} & 
 \co_{k}(U_{n}) \ar[d]_{\nu_{k}}
 \\
 \co_{k+1}(I_{n+1}) \ar[r] & 
 \co_{k+1}(U_{n+1}) \ar[r] &
 \co_{k}(V_{n+1}) \ar[r] & 
 \co_{k}(I_{n+1}) \ar[r] & 
 \co_{k}(U_{n+1})
 }
\end{align*}
of abelian groups with exact rows. We observe the following:
\begin{itemize}
 \item As $I$ is $\cofi{0}$-acyclic and $\weak(I) \leq r$, then $I \in \poly_{1}(r,0)$ and so for every $k \geq 0$ the map $\mu_{k}$ is an isomorphism for $n \geq 2k+r+1$ and a surjection for $n \geq 2k+r$.
 \item By the induction hypothesis on $U$, for every $k \geq 0$ the map $\nu_{k}$ is an isomorphism for $n \geq \max\{2k+2r-1,\,L\}$ and a surjection for $n \geq \max\{2k+2r-2,\,L\}$.
\end{itemize}
Therefore we have the following:
\begin{itemize}
 \item By the five-lemma, $\lambda_{k}$ is an isomorphism provided that $\nu_{k+1}$ and $\mu_{k}$ are isomorphisms, $\mu_{k+1}$ is surjective, and $\nu_{k}$ is injective: these are guaranteed in the range $n \geq \max\{2k+2r+1,\,L\}$ (noting $2(k+1) + r \leq 2(k+1) + 2r - 1$ because $r \geq 1$).
 
 \item By one of the four-lemmas, $\lambda_{k}$ is surjective provided that $\nu_{k+1}$ and $\mu_{k}$ are surjective, and $\nu_{k}$ is injective: these are guaranteed in the range $n \geq \max\{2k+2r,\,L\}$.
\end{itemize}
\end{proof}

\section{Application to congruence subgroups}
In this section we prove Theorem \ref{congruence-larger-action}.

\begin{proof}[Proof of \textbf{\emph{Theorem \ref{congruence-larger-action}}}]
By \cite[Theorem 4.15]{bahran-reg}, we have
\begin{itemize}
 \item $\weak(\co_{k}(\GL_{\bul}(R,I);\ab)) \leq 2k$, and
 \vspace{0.1cm}
 \item $\reg(\co_{k}(\GL_{\bul}(R,I);\ab)) \leq 
\begin{cases}
 2s+3 & \text{if $k=1$,}
 \\
 4k+2s & \text{if $k \geq 2$.}
\end{cases}
$
\end{itemize}
We now consider the groupoid $\gaga := \SL^{\fU}(R/I)$ in order to follow the argument and notation in \cite[proof of Theorem 1.4]{mpp-secondary}, with the following adjustment: Declare a new $\FI$-module $V$ via 
\begin{align*}
 V_{S} := 
\begin{cases}
\co_{k}(\GL_{S}(R,I);\ab) & \text{if $|S| \geq n_{0}$,}
 \\
 0 & \text{if $|S| < n_{0}$,}
\end{cases}
\end{align*}
so that by part (1) of Remark \ref{surj-rem}, $V$ extends to a $U\gaga$-module. Note that as an $\FI$-module by construction there is a short exact sequence 
\begin{align*}
 0 \rarr V \rarr \co_{k}(\GL_{S}(R,I);\ab) \rarr T \rarr 0
\end{align*}
with $\deg(T) \leq n_{0}-1 \leq 2s+2$. Invoking Corollary \ref{finite-support}, applying $\locoh{0}$ the associated long exact sequence here yields 
\begin{align*}
 h^{0}(V)
 &\leq h^{0}(\co_{k}(\GL_{\bul}(R,I);\ab)) \, ,
 \\
 h^{1}(V) &\leq \max\{\deg T,\, 
 h^{1}(\co_{k}(\GL_{\bul}(R,I);\ab))\} \, ,
 \\
 h^{j}(V) &= h^{j}(\co_{k}(\GL_{\bul}(R,I);\ab)) \quad \text{if $j \geq 2$.}
\end{align*}
Thus if $h^{j}(V) \geq 0$ for $j \neq 1$, we have $h^{j}(\co_{k}(\GL_{\bul}(R,I);\ab)) \geq 0$ and hence 
\begin{align*}
 h^{j}(V) + j \leq h^{j}(\co_{k}(\GL_{\bul}(R,I);\ab)) + j 
 \leq 
\begin{cases}
 2s+3 & \text{if $k=1$,}
 \\
 4k+2s & \text{if $k \geq 2$.}
\end{cases}
\end{align*}
by Theorem \ref{nss} applied to $\co_{k}(\GL_{\bul}(R,I);\ab)$. If $h^{1}(V) \geq 0$, there are two possibilities: 
\begin{itemize}
 \item $h^{1}(\co_{k}(\GL_{\bul}(R,I);\ab) \leq \deg T$. Then $h^{1}(V) + 1 \leq \deg T + 1 \leq 2s+3$.
 \item $h^{1}(\co_{k}(\GL_{\bul}(R,I);\ab) > \deg T$. Then $h^{1}(\co_{k}(\GL_{\bul}(R,I);\ab)) \geq 0$ and hence 
\begin{align*}
 h^{1}(V) + 1 \leq h^{1}(\co_{k}(\GL_{\bul}(R,I);\ab)) + 1 
 \leq
 \begin{cases}
 2s+3 & \text{if $k=1$,}
 \\
 4k+2s & \text{if $k \geq 2$.}
\end{cases}
\end{align*}
by Theorem \ref{nss}.
\end{itemize}
Applying Theorem \ref{nss} to $V$ now, we get
\begin{align*}
 \reg(V) \leq
\begin{cases}
 2s+3 & \text{if $k=1$,}
 \\
 4k+2s & \text{if $k \geq 2$.}
\end{cases}
\end{align*}
By \cite[Proposition 3.3]{cmnr-range} we also have $\weak(V) \leq 2k$.
Thus by Theorem \ref{thm:reg} and Remark \ref{compare-2}, $V$ has polynomial degree $\leq 2k$ 
\begin{align*}
 \text{in ranks }
 > \begin{cases}
 2s+3 & \text{if $k=1$,}
 \\
 4k+2s & \text{if $k \geq 2$,}
\end{cases}
\end{align*}
in the sense of \cite[Definition 2.40]{mpp-secondary}. \begin{birki}
 \item By \cite[Remark 2.42]{mpp-secondary}, $V$ has the same polynomial degree and rank bounds as a $U\gaga$-module.
 \item Noting that $\sr(R/I) \leq s$ as well \cite[Lemma 4.1]{bass-k-theory}, by \cite[Proposition 2.13]{mpp-secondary}, the category $U\gaga$ satisfies $\mathbf{H3}(2,s+1)$.
\end{birki}
Therefore by \cite[Theorem 3.11]{mpp-secondary}, we have 
\begin{align*}
 \wt{\co}_{i}^{\mathcal{G}}(V)_{n}
 = 0 \text{ for } n &> 
\begin{cases}
 \max\{2s+i+4,\, s+2i+3\} & \text{if $k = 1$,}
 \\
 \max\{4k+2s+i+1,\, 2k+s+2i+1\} & \text{if $k \geq 2$.}
\end{cases}
\end{align*}
and in particular
\begin{align*}
 \wt{\co}_{-1}^{\mathcal{G}}(V)_{n}
 = 0 \text{ for } n &> 
\begin{cases}
 2s + 3 & \text{if $k =1$,}
 \\
 4k + 2s & \text{if $k \geq 2$,}
\end{cases}
  \\
 \wt{\co}_{0}^{\mathcal{G}}(V)_{n}
  = 0 \text{ for } n &> 
\begin{cases}
 2s + 4 & \text{if $k =1$,}
 \\
 4k + 2s + 1 & \text{if $k \geq 2$.}
\end{cases}
\end{align*}
Noting that the definitions of $\wt{\co}^{\mathcal{G}}_{*}$ in \cite[Definition 3.14]{mpw-torelli-H2} and \cite[Definition 2.9]{mpp-secondary} are consistent with each other, the vanishing above corresponds to a coequalizer diagram of the form 
\begin{align*}
 \Ind_{\gaga_{n-2}}^{\gaga_{n}} V_{n-2} \rightrightarrows   
 \Ind_{\gaga_{n-1}}^{\gaga_{n}} V_{n-1}
 \rarr
 V_{n}
\end{align*}
of $\zz\gaga_{n}$-modules whenever 
\begin{align*}
 n \geq 
\begin{cases}
 2s + 5 & \text{if $k =1$,}
 \\
 4k + 2s + 2 & \text{if $k \geq 2$.}
\end{cases}
\end{align*}
by \cite[Remark 3.16]{mpw-torelli-H2}. In this range, we have $n-2 \geq 2s+3 \geq n_{0}$, so that 
\begin{align*}
 V_{j} = \co_{k}(\GL_{j}(R,I);\ab) \quad \text{for $j \in \{n-2,\, n-1,\, n\}$} \, .
\end{align*}
\end{proof}

\bibliographystyle{hamsalpha}
\bibliography{stable-boy}

\end{document}